\documentclass[11pt, a4paper]{article}
\usepackage{times}
\usepackage{a4wide}
\usepackage[dvips, hyperindex]{hyperref}
\usepackage[british]{babel}
\usepackage{enumerate}
\usepackage{amsmath, amscd, amsfonts, amsthm, amssymb, latexsym, comment, stmaryrd, graphicx}
\usepackage[T1]{fontenc}
\usepackage[latin1]{inputenc}

\newtheorem{thm}{Theorem}[section]
\newtheorem{lem}[thm]{Lemma}

\newtheorem{prop}[thm]{Proposition}
\newtheorem{cor}[thm]{Corollary}

\DeclareMathOperator{\GL}{GL}
\DeclareMathOperator{\PGL}{PGL}
\DeclareMathOperator{\PSL}{PSL}
\DeclareMathOperator{\PXL}{PXL}
\DeclareMathOperator{\Gal}{Gal}
\DeclareMathOperator{\Frob}{Frob}
\DeclareMathOperator{\Stab}{Stab}

\newcommand{\proj}{\mathrm{proj}}

\newcommand{\cA}{\mathcal{A}}
\newcommand{\cB}{\mathcal{B}}

\newcommand{\fp}{\mathfrak{p}}
\newcommand{\fP}{\mathfrak{P}}

\newcommand{\CC}{\mathbb{C}}
\newcommand{\FF}{\mathbb{F}}
\newcommand{\NN}{\mathbb{N}}
\newcommand{\QQ}{\mathbb{Q}}
\newcommand{\RR}{\mathbb{R}}

\newcommand{\Qbar}{\overline{\QQ}}
\newcommand{\Fbar}{\overline{\FF}}

\newcommand{\vect}[2]{
 \left(  \begin{matrix} #1 \\ #2 \end{matrix} \right)}

\begin{document}

\title{An Application of Maeda's Conjecture\\ to the Inverse Galois Problem}
\author{Gabor Wiese\footnote{Université du Luxembourg,
Faculté des Sciences, de la Technologie et de la Communication,
6, rue Richard Coudenhove-Kalergi,
L-1359 Luxembourg, Luxembourg, gabor.wiese@uni.lu}}
\maketitle

\begin{abstract}
It is shown that Maeda's conjecture on eigenforms of level~$1$ implies
that for every positive even $d$ and every $p$ in a density-one set of primes,
the simple group $\PSL_2(\FF_{p^d})$ occurs as the Galois group of
a number field ramifying only at~$p$.

\end{abstract}

\section{Introduction}

The purpose of this paper is to support the approach to the inverse Galois problem
for certain finite groups of Lie type through automorphic forms. There have been a number
of promising results in the recent past, e.g.\ \cite{DiWi} and \cite{SL2} for groups
of the type $\PSL_2(\FF_{\ell^d})$, and \cite{KLS}, \cite{KLS2}, \cite{PartIII}
for more general groups.
The general idea is to take varying automorphic forms over~$\QQ$ and to study the images
of the residual Galois representations attached to them.
Currently one only obtains positive-density or infinity results.
The main technical obstacle to improving the mentioned results to density~$1$
seems to be the lack of control on the fields of coefficients of the automorphic forms
involved.

In the easiest case, that of `classical' modular forms, i.e.\ of automorphic forms
for $\GL_2$ over~$\QQ$, there is a strong conjecture due to Maeda on the coefficient
fields of level~$1$ modular forms.
In order to demonstrate the potential of the modular approach to the inverse Galois problem,
we show that the control on the coefficient fields provided by Maeda's conjecture
suffices to yield the following strong result on the inverse Galois problem.

\begin{thm}\label{thm:main}
Assume the following form of {\em Maeda's conjecture} on level~$1$ modular forms:
\begin{quote}
For any $k$ and any normalised eigenform $f \in S_k(1)$ (the space of cuspidal modular
forms of weight~$k$ and level~$1$), the coefficient field
$\QQ_f := \QQ(a_n(f) \;|\; n \in \NN)$
has degree equal to $d_k := \dim_\CC S_k(1)$ and the Galois group of its normal closure over~$\QQ$
is the symmetric group $S_{d_k}$.
\end{quote}

\begin{enumerate}[(a)]
\item Let $2 \le d \in \NN$ be even.
Then the set of primes $p$ such that there is a number field $K/\QQ$
ramified only at~$p$ with Galois group isomorphic to $\PSL_2(\FF_{p^d})$
has density~$1$.
\item Let $1 \le d \in \NN$ be odd.
Then the set of primes $p$ such that there is a number field $K/\QQ$
ramified only at~$p$ with Galois group isomorphic to $\PGL_2(\FF_{p^d})$
has density~$1$.
\end{enumerate}
\end{thm}

Maeda's conjecture was formulated as Conjecture~1.2 in~\cite{HM}. It has
been checked up to weight 12000 (see~\cite{GMcA}).
We also mention that a generalisation of a weaker form of Maeda's conjecture to
squarefree levels has recently been proposed by Tsaknias~\cite{Tsaknias}.

Throughout the paper the notion of density can be taken to be either natural density
or Dirichlet density.

It is certainly possible to give an effective version of Theorem~\ref{thm:main}.
Suppose that Maeda's conjecture has been checked for weights up to~$B$.
Then for all $d \le \dim_\CC S_B(1)$ one can work out an explicit lower bound for
the density of the sets in the theorem, depending on~$B$.

The proof of Theorem~\ref{thm:main} is given in the remainder of the paper.
It is based on a meanwhile classical `big image result' of Ribet~\cite{Ri},
Chebotarev's density theorem, some combinatorics in symmetric groups,
and Galois theory.

\subsection*{Acknowledgements}

I wish to thank Sara Arias-de-Reyna for reading an earlier version of this
paper and her corrections and valuable remarks. I also thank J\"urgen Kl\"uners
for the reference to van der Waerden.
I acknowledge partial support by the Priority Program 1489 of the Deutsche
Forschungsgemeinschaft (DFG) and
by the Fonds National de la Recherche Luxembourg (INTER/DFG/12/10).

\section{Proof}

In this section the main result is proved. We use the convention that
the symmetric group $S_n$ is the group of permutations of the set $\{1,2,\dots,n\}$.

\subsection*{Splitting of primes in extensions with symmetric Galois group}

In this part we give a possibly nonstandard proof of the well-known fact that the
splitting behaviour of unramified primes in a simple extension $K(a)/K$
can be read off from the cycle type of the Frobenius, seen as an element of
the permutation group of the roots of the minimal polynomial of~$a$.
(A more `standard' proof would consider the factorisation into irreducibles
of the reduction of the minimal polynomial of~$a$, as in~\cite{vdW}, p.~198).

Let $M/K$ be a separable field extension of degree~$n$ and let $L/M$ be
the Galois closure of $M$ over~$K$. By the theorem of
the primitive element there is $a \in M$ such that $M = K(a)$. Let $f \in K[X]$ be
the minimal polynomial of~$a$ over~$K$ and let $a=a_1,a_2,\dots,a_n$ be the roots of~$f$ in~$L$.
The map $\psi: G := \Gal(L/K) \to S_n$, sending $\sigma$ to the permutation $\psi(\sigma)$
given by $\sigma(a_i) = a_{\psi(\sigma)(i)}$ is an injective group homomorphism,
which maps $H := \Gal(L/M)$ onto $\Stab_{S_n}(1) \cap \psi(G)$.

\begin{prop}\label{prop:nt}
Assume the preceding set-up with $K$ a number field.
Let $\fp$ be a prime of~$K$ and $\fP$ a prime of~$L$ dividing~$\fp$.
We suppose that $\fP/\fp$ is unramified.
Then the cycle lengths in the cycle decomposition of $\psi(\Frob_{\fP/\fp}) \in S_n$
are precisely the residue degrees of the primes of~$M$ lying above~$\fp$.
\end{prop}

\begin{proof}
Let $g \in \Gal(L/K)$. Denote by $\Frob_{g\fP/\fp}$ the Frobenius element
of $g\fP/\fp$ in $\Gal(L/K)$
and by $f_{(g\fP\cap M)/\fp}$ the inertial degree of the
prime $g\fP \cap M$ of $M$ over~$\fp$.
Write $\varphi := \Frob_{\fP/\fp}$ for short. We have
$$ f_{(g\fP\cap M)/\fp} = \min_{i \in \NN} (\Frob_{g\fP/\fp}^i \in H)
= \min_{i \in \NN} (\varphi^i \in g^{-1}Hg).$$
From this we obtain the equivalences:
\begin{align*}
&\exists g \in G: \; f_{(g\fP \cap M)/\fp} = d \\
\Leftrightarrow\;\;\;
&\exists g \in G: \; \varphi^d \in g^{-1}Hg \textnormal{ and } \forall\, 1 \le i < d: \varphi^i \not\in g^{-1}Hg \\
\Leftrightarrow\;\;\;
&\exists g \in G: \; \psi(\varphi^d) \in \Stab_{S_n}(\psi(g^{-1})(1))
\textnormal{ and } \forall\, 1 \le i < d: \psi(\varphi^i) \not\in\Stab_{S_n}(\psi(g^{-1})(1))\\
\Leftrightarrow\;\;\;
&\exists j \in \{1,\dots,n\}: \; \psi(\varphi^d) \in \Stab_{S_n}(j)
\textnormal{ and } \forall\, 1 \le i < d: \psi(\varphi^i) \not\in\Stab_{S_n}(j)\\
\Leftrightarrow\;\;\;
&\psi(\varphi)\textnormal{ contains a $d$-cycle.}
\end{align*}
This proves the proposition.
\end{proof}

\subsection*{Combinatorics in symmetric groups}

We will eventually be interested in primes of a fixed residue degree~$d$ in an
extension with symmetric Galois group. The results of the previous part hence
lead us to consider elements in symmetric groups having a $d$-cycle,
which we do in this part.

The contents of this part is presumably also well-known. Since the techniques are very simple and
straight forward, I decided to include the proofs rather than to look for suitable references.
Let $d \ge 1$ be a fixed integer.
Define recursively for $i \ge 1$ and $1 \le j \le i$
$$ a(0):= 0, \;\;\; b(i,j) := \frac{1}{j! d^j}(1-a(i-j)), \;\;\; a(i) := \sum_{k=1}^i b(i,k).$$

\begin{lem}\label{lem:seq-Sn}
With the preceding definitions we have
$$a(i) = \sum_{j=1}^i \frac{(-1)^{j+1}}{j!d^j}
= 1- \exp(\frac{-1}{d}) + \sum_{j=i+1}^\infty \frac{(-1)^j}{j!d^j}.$$
\end{lem}

\begin{proof}
This is a simple induction. For the convenience of the reader, we include
the inductive step:
\begin{align*}
a(i+1)
&= \sum_{k=1}^{i+1} b(i+1,k)
= \sum_{k=1}^{i+1} \frac{1}{k!d^k} (1 - a(i+1-k))\\
&= \sum_{k=1}^{i+1} \frac{1}{k!d^k} \left(1 - \sum_{j=1}^{i+1-k} \frac{(-1)^{j+1}}{j!d^j}\right)
= \sum_{k=1}^{i+1}\left(\frac{1}{k!d^k} - \sum_{j=1}^{i+1-k}\frac{(-1)^{j+1}}{k!j!d^{j+k}}\right)\\
&= \sum_{m=1}^{i+1}\frac{1}{m!d^m} + \sum_{m=2}^{i+1} \frac{1}{m!d^m} \sum_{j=1}^{m-1} \vect mj (-1)^j
= \sum_{m=1}^{i+1}\frac{(-1)^{m+1}}{m!d^m}.
\end{align*}
\end{proof}

For $i \to \infty$ the convergence $a(i) \to 1- \exp(\frac{-1}{d})$
is very quick because of the simple estimate of the error term
$\left| \sum_{j=i+1}^\infty \frac{1}{j^! d^j} \right | \le \frac{2}{(i+1)! d^{i+1}}$.

We now relate the quantities $a(i)$ and $b(i,j)$ to proportions in the symmetric group.
Let $n,j \in \NN$. Define
\begin{align*}
\cA_n(d) &:= \{ g \in S_n \;|\; g \textnormal{ contains at least one $d$-cycle }\},\\
\cB_n(d,j) &:= \{ g \in S_n \;|\; g \textnormal{ contains precisely $j$ $d$-cycles }\}.
\end{align*}

\begin{lem}\label{lem:Sn}
For all $n \ge 2d$ the following formulae hold, where $i := \lfloor \frac{n}{d} \rfloor$:
\begin{enumerate}[(a)]
\item\label{lem:Sn:a} $n! \cdot a(i) = \#\cA_n(d)$
\item\label{lem:Sn:b} $n! \cdot b(i,j) = \# \cB_n(d,j)$
\item\label{lem:Sn:c} $n! \cdot \frac{2n-d-1}{n(n-1)} (1-a(i-1))= \#\{g \in \cB_n(d,1) \;|\;
\textnormal{ the unique $d$-cycle contains $1$ or $2$}\}$.
\item\label{lem:Sn:d} $n! \cdot \frac{1}{n(n-1)} (1- a(i-2)) = \#\{g \in \cB_n(d,2) \;|\;
\textnormal{ one $d$-cycle contains $1$, the other $2$}\}$.
\end{enumerate}
\end{lem}

\begin{proof}
\eqref{lem:Sn:a} and \eqref{lem:Sn:b} are proved by induction for $n \ge 1$.
For $n < d$ (i.e.\ $i=0$), the equalities are trivially true. Now we describe the induction step:
\begin{align*}
\#\cB_n(d,j)
&= \frac{1}{j!} \cdot \left(\vect nd \cdot (d-1)!\right) \cdot \left(\vect {n-d}d \cdot (d-1)!\right)
\cdot \ldots \cdot \left(\vect {n-(j-1)d}d \cdot (d-1)!\right)\\
& \;\;\;\;\;\; \cdot (n-jd)! \cdot (1-a(i-j))\\
&= \frac{n!}{j! d^j}(1-a(i-j)).
\end{align*}
The first equality can be seen as follows: There are $j!$ ways of ordering the $j$ $d$-cycles.
The number of choices for the first $d$-cycle is given by $\vect nd \cdot (d-1)!$,
the one for the second is $\vect {n-d}d \cdot (d-1)!$, and so on. After having chosen
$j$, $d$-cycles $n-jd$ elements remain. Among these remaining elements we may only take
those that do not contain any $d$-cycle; their number is
$(n-jd)!\cdot (1-a(i-j))$ by induction hypothesis.

\eqref{lem:Sn:c} The number of elements in the set in question is
$$ \left(2 \vect{n-1}{d-1} - \vect{n-2}{d-2}\right)(d-1)! \cdot (n-d)!\cdot(1-a(i-1))
= n! \frac{2n-d-1}{n(n-1)} (1-a(i-1))$$
because $\vect{n-1}{d-1} \cdot (d-1)!$ is the number of choices for a $d$-cycle with
one previously chosen element (i.e.\ $1$ or $2$) and $\vect{n-2}{d-2}\cdot(d-1)!$
is the number of choices for a $d$-cycle containing $1$ and~$2$.

\eqref{lem:Sn:d} The number of elements in the set in question is
$$ \vect{n-2}{d-1}(d-1)! \cdot \vect{n-2-(d-1)}{d-1}(d-1)! \cdot (n-2d)!\cdot (1-a(i-2))
= n! \frac{1}{n(n-1)} (1-a(i-2))$$
because $\vect{n-2}{d-1} \cdot (d-1)!$ is the number of choices for a $d$-cycle
containing~$1$ and not containing~$2$ and $\vect{n-2-(d-1)}{d-1}\cdot (d-1)!$
is the number of choices for a $d$-cycle containing~$2$ among the elements remaining
after the first choice, and again $(n-2d)!\cdot (1-a(i-2))$ is the number of elements
remaining after the two choices such that they do not contain any $d$-cycle.
\end{proof}

We write $\cA_n^\pm(d)$ for the subsets of $\cA_n(d)$ consisting of the elements
having positive or negative signs.

\begin{cor}\label{cor:Sn}
Let $d,n \in \NN$, $n \ge 2d \ge 2$ and put $i := \lfloor \frac{n}{d} \rfloor$. Then the estimates
$$\left| \#\cA_n^+(d) - \#\cA^-_n(d) \right| \le n! \cdot \big( \frac{2n-d-1}{n(n-1)}(1-a(i-1))
+ \frac{1}{n(n-1)}(1-a(i-2)) \big) \le n! \cdot \frac{2}{n-1}$$
and
$$\left| \frac{\#\cA_n^+(d) - \#\cA_n^-(d)}{\#\cA_n(d)} \right|
\le \frac{1}{n-1} \cdot \frac{2}{1 - \exp(-\frac{1}{d}) - \frac{2}{(i+1)! d^{i+1}}}$$
hold.
\end{cor}

\begin{proof}
Consider the bijection $\phi: S_n \xrightarrow{g \mapsto g \circ (1 2)} S_n$.
For $j > 2$ the image of $\cA_n^+(d) \cap \cB_n(d,j)$ under $\phi$ lands in $\cA_n^-(d)$ because
the multiplication with $(1 2)$ can at most remove two $d$-cycles.
Now consider $g \in \cA_n^+(d) \cap \cB_n(d,2)$. Clearly $\phi(g) \in \cA_n^-(d)$ unless
one of the $d$-cycles contains~$1$ and the other one contains~$2$.
For $g \in \cA_n^+(d) \cap \cB_n(d,1)$ we find that $\phi(g) \in \cA_n^-(d)$ unless
the single $d$-cycle of~$g$ contains $1$ or~$2$.
In view of Lemma~\ref{lem:Sn} we thus obtain the inequality
$$ \#\cA_n^+(d) - \#\cA_n^-(d) \le n! \cdot \big( \frac{2n-d-1}{n(n-1)}(1-a(i-1))
+ \frac{1}{n(n-1)}(1-a(i-2)) \big) \le n! \cdot \frac{2}{n-1}.$$
By exchanging the roles of $+$ and $-$ we obtain the first estimate.
The second estimate then is an immediate consequence of Lemma~\ref{lem:seq-Sn} and the
trivial estimate of the error term mentioned after it.
\end{proof}

\subsection*{Density of primes with prescribed residue degree
in composite of field extensions with symmetric Galois groups}

\begin{lem}\label{lem:twofields}
Let $1 \le d \in \NN$, $K$ be a field and $L/K$, $F/K$ be two finite Galois extensions such that
$\Gal(L/K) \cong S_n$ with $n \ge \max(5,2d)$ and $L$ is not a subfield of~$F$.
Let $C \subseteq G:= \Gal(F/K)$ be a subset and put $c := \frac{\#C}{\#G}$
and $a := \frac{\#\cA_n(d)}{\#S_n} = a(\lfloor \frac{n}{d} \rfloor)$.

Let $X := \Gal(LF/K)$ and $Y$ be the subset of~$X$ consisting of those elements that
project to an element in $\cA_n(d) \subseteq S_n \cong \Gal(L/K)$
or to an element in~$C \subseteq \Gal(F/K)$ under the natural projections. Then
$$ \frac{\#Y}{\#X} = a + c - (1+\delta) ac, \textnormal{ where }
\begin{cases} \delta = 0 & \textnormal{ if } L \cap F = K,\\
|\delta| \le \frac{1}{n-1} \cdot \frac{2}{1 - \exp(-\frac{1}{d}) - \frac{2}{(1+\lceil \frac{n}{d} \rceil)! d^{1+\lceil \frac{n}{d} \rceil}}}  & \textnormal{ otherwise.}\end{cases}$$
\end{lem}

\begin{proof}
The intersection $L \cap F$ is a Galois extension of~$K$ which is contained in~$L$.
The group structure of $S_n$ (more precisely, the fact that the alternating group~$A_n$
is the only nontrivial normal subgroup of~$S_n$)
hence implies that $[L \cap F:K] \le 2$; for, if $L\cap F$
were equal to~$L$, then $L$ would be a subfield of~$F$, which is excluded by assumption.

Assume first $L \cap F = K$, then $\Gal(LF/K) \cong \Gal(L/K) \times \Gal(F/K)$ and thus
$$ \#Y = \#\cA_n(d) \cdot \#G + \#S_n \cdot \#C - \#\cA_n(d) \cdot \#C,$$
from which the claimed formula follows by dividing by $\#X = \#G \cdot \# S_n$.

Assume now that $L \cap F =: N$ is a quadratic extension of~$K$. Then $X$ is isomorphic to
the index~$2$ subgroup of $\Gal(L/K) \times \Gal(F/K)$ consisting of those pairs
of elements $(g,h)$ such that $g$ and $h$ project to the same element in $\Gal(N/K)$.
The elements of
$\cA_n(d)$ that project to the identity of $\Gal(N/K)$ are precisely those in $\cA_n^+(d)$.
In a similar way we denote by $C^+$ those elements of $C$ projecting to the identity
of $\Gal(N/K)$, and by $C^-$ the others. Then we have
$$\#Y = \#\cA_n(d) \cdot \frac{\#G}{2} + \frac{\# S_n}{2} \cdot \#C -
\#\cA_n^+(d) \cdot \#C^+ - \#\cA_n^-(d) \cdot \#C^-.$$
Dividing by $\#X = \frac{\#S_n \cdot \#G}{2}$ we obtain
$$\frac{\#Y}{\#X} = a + c - (1+\delta)ac,
\textnormal{ where }
\delta = \frac{\#C^+ - \#C^-}{\#C} \cdot \frac{\#\cA_n^+(d) - \#\cA_n^-(d)}{\#\cA_n(d)}.$$
The claim is now a consequence of Corollary~\ref{cor:Sn}.
\end{proof}

\begin{lem}\label{lem:seq}
Let $(a_n)_{n\ge 1}$ be a sequence of non-negative real numbers such that $\sum_{n=1}^\infty a_n$ diverges.
\begin{enumerate}[(a)]
\item \label{part:seq:a}
Let $\gamma >0$, $b_0 \in \RR$. Assume $a_n < \frac{1}{\gamma}$ for all $n\ge 1$.
We define a sequence $(b_n)_{n \ge 0}$ by the rule
$$ b_n := b_{n-1} + a_n - \gamma b_{n-1} a_n$$
for all $n\ge 1$.
Then the sequence $(b_n)_{n \ge 1}$ tends to~$1/\gamma$ for
$n \to \infty$.

\item \label{part:seq:b}
Let $(\delta_n)_{n \ge 1}$ be a sequence of real numbers tending to~$0$ and
let $c_0 \in \RR$.
Assume $\limsup_{n\to \infty} a_n < 1$.
We define the (modified) {\em inclusion-exclusion sequence} as
$$ c_n := c_{n-1} + a_n - (1+\delta_n) c_{n-1} a_n \textnormal{ for } n \ge 1.$$
Then the sequence $(c_n)_{n \ge 1}$ tends to~$1$.
\end{enumerate}
\end{lem}

\begin{proof}
\eqref{part:seq:a} We let
\begin{multline*}
r_n := 1 - \gamma b_n
= 1 - \gamma (b_{n-1} + a_n - \gamma b_{n-1} a_n)
= (1 - \gamma b_{n-1})(1-\gamma a_n)\\
= (1-\gamma b_0) (1- \gamma a_1) (1- \gamma a_2) \cdots (1- \gamma a_n).
\end{multline*}
To see that the limit of $(\gamma b_n)_{n \ge 0}$ is~$1$, we take the logarithm
of $(1- \gamma a_1) (1- \gamma a_2) \cdots (1- \gamma a_n)$:
$$ \sum_{i=1}^n \log(1- \gamma a_i)
= - \gamma \sum_{i=1}^n a_i - \sum_{i=1}^n \sum_{j=2}^\infty \frac{(\gamma a_i)^j}{j}
\le - \gamma \sum_{i=1}^n a_i.$$
By our assumption this diverges to $-\infty$ for $n\to \infty$, so that
$\lim_{n \to \infty} r_n = 0$, proving the lemma.

\eqref{part:seq:b}
Let $\min(1, \frac{1}{\limsup_{n\to \infty} a_n} - 1) > \epsilon > 0$.
There is $N$ such that $|\delta_n| < \epsilon$ and $a_n <\frac{1}{1+\epsilon}$ for all $n\ge N$.
By enlarging $N$ if necessary we may also assume $c_N \ge 0$.
The reason for the latter is that $c_{N+n} > c_N + \sum_{i=1}^n a_{N+i}$
if $c_{N+i} < 0$ for all $0 \le i \le n$.

We consider the two sequences
$$ b_N := c_N \textnormal{ and }
   b_n = b_{n-1} + a_n - (1+\epsilon) b_{n-1} a_n \textnormal{ for } n > N$$
and
$$ d_N := c_N \textnormal{ and }
   d_n = d_{n-1} + a_n - (1-\epsilon) d_{n-1} a_n \textnormal{ for } n > N.$$
By \eqref{part:seq:a} we know $\lim_{n\to \infty} b_n = \frac{1}{1+\epsilon}$ and
$\lim_{n\to \infty} d_n = \frac{1}{1-\epsilon}$.
For $n \ge N$ by induction we obtain the estimate:
$$ 0 \le b_n \le c_n \le d_n.$$
Thus, there is $M$ such that $\frac{1}{1+\epsilon}-\epsilon \le c_n \le \frac{1}{1-\epsilon} + \epsilon$
for all $n \ge M$.
As $\epsilon$ is arbitrary, we find $\lim_{n \to \infty} c_n = 1$.
\end{proof}

\begin{prop}\label{prop:fields}
Let $1 \le d \in \NN$, $K$ be a field and let $L_n$ for $n \in \NN$ be Galois extensions of~$K$ with
Galois group $\Gal(L_n/K) \cong S_{N_n}$ such that $N_n < N_{n+1}$ for all $n \ge 1$.
Denote by $G_n$ the Galois group of the composite field $L_1L_2\cdots L_n$ over~$K$
and for $1 \le i \le n$ denote by $\pi_i: G_n \to \Gal(L_i/K)$ the natural projection. Consider
$$ c_n := \frac{\#\{g \in G_n \;|\; \exists i \in \{1,\dots,n\}:
\pi_i(g) \in \Gal(L_i/K) \cong S_{N_i} \textnormal{ contains at least one $d$-cycle } \}}{\# G_n}.$$

Then the sequence $c_n$ tends to~$1$ for $n\to \infty$.
\end{prop}

\begin{proof}
Without loss of generality we can assume $\max(5,2d) \le N_1$.
Let $c_0 := 0$ and $a_n := a(\lfloor \frac{N_n}{d} \rfloor) = \frac{\#\cA_{N_n}(d)}{\# S_{N_n}}$.
By Lemmas \ref{lem:seq-Sn} and~\ref{lem:Sn} it is clear that $\sum_{n=1}^\infty a_n$ diverges.

If we call $K_i$ the unique quadratic extension of~$K$ inside~$L_i$, then
Lemma~18.3.9 of~\cite{FJ} shows that $\Gal(L_1\cdots L_n/K_1\cdots K_n) \cong
A_{N_1} \times \dots \times A_{N_n}$ for all~$n\ge 1$. This implies that $L_n$
cannot be a subfield of $L_1\cdots L_{n-1}$ for any $n \ge 2$.

Lemma~\ref{lem:twofields} inductively
gives the formula $c_n = a_n + c_{n-1} - (1+\delta_n)a_nc_{n-1}$
for $n \ge 1$, where $\delta_n$ is bounded by
$$|\delta_n| \le \frac{1}{N_n-1} \cdot \frac{2}{1 - \exp(-\frac{1}{d}) - \frac{2}{(1+\lceil \frac{N_n}{d} \rceil)! d^{1+\lceil \frac{N_n}{d} \rceil}}},$$
which clearly tends to~$0$ for $n \to \infty$. Lemma~\ref{lem:seq} yields the claim on the limit.
\end{proof}

By applying Chebotarev's density theorem and noting that the set in the proposition
is conjugation invariant, we obtain the following corollary.

\begin{cor}\label{cor:fields}
Let $1 \le d \in \NN$, $K$ be a number field and let $L_n$ for $n \in \NN$ be Galois extensions of~$K$ with
Galois group $\Gal(L_n/K) \cong S_{N_n}$ such that $N_n < N_{n+1}$ for all $n \ge 1$.

Then the set of primes of~$K$
$$\{ \fp \;|\; \exists i \in \{1,\dots,n\}:
\pi_i(\Frob_\fp) \in \Gal(L_i/K) \cong S_{N_i} \textnormal{ contains at least one $d$-cycle } \}$$
has a density, and the density is equal to~$c_n$ from Proposition~\ref{prop:fields}
and hence tends to~$1$ for $n \to \infty$.
Here $\Frob_\fp = \Frob_{\fP/\fp}$ for any prime~$\fP$
of the composite field $L_1L_2\cdots L_n$ above~$\fp$.
\end{cor}

The following is the main theorem of this paper concerning the density of primes
with prescribed residue degree in a composite of field extensions with symmetric Galois groups.

\begin{thm}\label{thm:general}
Let $1 \le d \in \NN$, $K$ be a number field and let $M_n$ for $n \in \NN$ be field extensions of~$K$ with
splitting field $L_n$ over~$K$ having Galois group $\Gal(L_n/K) \cong S_{N_n}$ such that
$N_n < N_{n+1}$ for all $n \ge 1$.

Then the set of primes of~$K$
$$\{ \fp \;|\; \exists i \in \{1,\dots,n\}, \exists \fP/\fp
\textnormal{ prime of $M_i$ of residue degree }d \}$$
has a density, and the density is equal to~$c_n$ from Proposition~\ref{prop:fields}
and hence tends to~$1$ for $n \to \infty$.
\end{thm}

\begin{proof}
Because of Proposition~\ref{prop:nt} the set of primes in the theorem is the same
as the set in Corollary~\ref{cor:fields}.
\end{proof}

\subsection*{End of the proof}

\begin{proof}[Proof of Theorem~\ref{thm:main}.]
Since $\dim_\CC S_k(1)$ tends to~$\infty$ for $k \to \infty$ (for even~$k$),
Maeda's conjecture implies the existence of newforms $f_n$ of level one
and increasing weight (automatically without complex multiplication because
of level~$1$)
such that their coefficient fields $M_n := \QQ_{f_n}$ satisfy the assumptions of
Theorem~\ref{thm:general}.

For each~$n$ and each prime $\fP$ of $M_n$ consider the Galois representation
$\rho_{f_n,\fP}^\proj: \Gal(\Qbar/\QQ) \to \PGL_2(\Fbar_p)$ attached to~$f_n$.
Theorem~3.1 of Ribet~\cite{Ri} implies that for each $f_n$ and all but possibly
finitely many $\fP$, its image is equal to $\PGL_2(\FF_\fP)$,
if the residue field $\FF_\fP$ of $\fP$ has odd degree over its prime field,
and equal to $\PSL_2(\FF_\fP)$ if the residue degree is even.
We will abbreviate this by $\PXL_2(\FF_\fP)$.

Consequently, the set of primes (of~$\QQ$)
$$\{ p \;|\; \exists i \in \{1,\dots,n\}, \exists \fP/p
\textnormal{ prime of $M_i$ s.t. } \rho_{f_i,\fP}^\proj \cong \PXL_2(\FF_{p^d}) \}$$
has the same density as the corresponding set in Theorem~\ref{thm:general},
implying Theorem~\ref{thm:main}.
\end{proof}

\bibliography{MIGP}
\bibliographystyle{amsalpha}

\end{document}